\documentclass[sn-mathphys]{sn-jnl}

\usepackage{amsmath}
\usepackage{algorithm}
\usepackage{algorithmicx}
\usepackage{algpseudocode}
\usepackage{amsthm}
\usepackage{mathrsfs}
\usepackage{amsfonts}
\usepackage{indentfirst}
\usepackage{caption}
\allowdisplaybreaks[4]

\jyear{2022}%

\newtheorem{thm}{Theorem}
\newtheorem{lem}{Lemma}

\newtheorem{definition}{Definition}
\newtheorem{coroll}{Corollary}
\newtheorem{remark}{Remark}
\newtheorem{assumption}{Assumption}

\begin{document}

\title[On sampling Kaczmarz-Motzkin methods for solving large-scale nonlinear systems]{On sampling Kaczmarz-Motzkin methods for solving large-scale nonlinear systems}


\author{\fnm{Feiyu} \sur{Zhang}}\email{s20090016@s.upc.edu.cn}

\author*{\fnm{Wendi} \sur{Bao}*}\email{baowendi@sina.com}

\author{\fnm{Weiguo} \sur{Li}}\email{liwg@upc.edu.cn}
\equalcont{These authors contributed equally to this work.}

\author{\fnm{Qin} \sur{Wang}}\email{s20090002@s.upc.edu.cn}
\equalcont{These authors contributed equally to this work.}

\affil{\orgdiv{College of Science}, \orgname{China University of Petroleum}, \city{Qingdao}, \postcode{266580}, \country{P.R.China}}





\abstract{In this paper, for solving large-scale nonlinear equations we first propose a nonlinear sampling Kaczmarz-Motzkin (NSKM) method. Based on the local tangential cone condition and the Jensen's inequality, we prove convergence of our method with two different assumptions. Then, for solving nonlinear equations with the convex constraints we propose two variants of the NSKM method: the projected sampling Kaczmarz-Motzkin (PSKM) method and the accelerated projected sampling Kaczmarz-Motzkin (APSKM) method. With the use of the nonexpansive property of the projection and the convergence of the NSKM method, the convergence analysis is obtained. Numerical results show that the NSKM method with the sample of the suitable size outperforms the nonlinear randomized Kaczmarz (NRK) method in terms of calculation times. The APSKM and PSKM methods are practical and promising for the constrained nonlinear problem. }




\maketitle

\section{Introduction}\label{intro}
Consider the nonlinear equations with finite convex constraints
\begin{equation}\label{equ1}
	f(x)=0\quad \text{subject to} \quad x\in C,
\end{equation}
where $f:\mathscr{D}(f)\subseteq\mathbb{R}^n \to \mathbb{R}^m$ is a nolinear vector-valued function, $x\in \mathbb{R}^n$ is an unknown vector and $C$ is a nonempty intersection of finite nonempty closed convex sets $C_i,$ i.e. $C=\bigcap \limits_{i=1}^{k_c} C_i$ ($k_c$ is potentially a large number). If $x^*\in \mathbb{R}^n$ exists such that $f(x^*)=0$ and $x^*\in C$, then $x^*$ is a solution of \eqref{equ1}. Such problems arise from many practical applications, e.g., electrical impedance tomography, circuit problems and chemical equilibrium systems.

When $C=\mathbb{R}^n$, the system \eqref{equ1} is a unconstrained problem. This problem has attracted widespread attention. Many computational methods have been proposed, for example, Newton method \cite{kelley1995iterative}, Quasi-Newton method \cite{dennis1977quasi}, Gauss-Newton method \cite{li1999globally} and Levenberg-Marquardt method \cite{yamashita2001rate}. These methods require the information of the whole nonlinear system, which may needs a large amount of computation for solving a large-scale nonlinear system. Stochastic gradient descent (SGD) method \cite{jin2020convergence} requires only evaluating one randomly selected nonlinear equation at each iteration, instead of the whole nonlinear system, which substantially reduces the computational cost per iteration and enables excellent to deal with the large-scale problems. Recently, Wang and Li et al. \cite{wang2022nonlinear} proposed a class of randomized Kaczmarz algorithms for solving large-scale nonlinear equations with specific assumptions, which only needs to calculate one row of the Jacobian matrix instead of the entire Jacobian matrix at each iteration and reduced the amount of calculation and storage. In addition, numerical results showed that algorithms proposed in \cite{wang2022nonlinear} are superior to the SGD algorithm. However, in the nonlinear randomized Kaczmarz (NRK) method \cite{wang2022nonlinear}, to determine the row index of the Jacobian matrix, all entries of $f(x)$ need to be calculated at each step, which is expensive and inefficient when the size of $f(x)$ is very large. In order to overcome the problem, Needell et al. \cite{de2017sampling} presented the sampling Kaczmarz-Motzkin (SKM) method for solving large-scale systems of linear inequalities, which only needs to compute a portion of the residuals at each step.

In general, the system \eqref{equ1} is a large-scale nonlinear problems with a large number of convex constraints. Based the SGD method, Wang \cite{wang2013incremental} studied the stochastic gradient descent method with a single random projection (PSGD). They randomly picked one out of all constraint sets and found the projection onto it after using stochastic gradient descent at each iteration. Meanwhile,  using a linear combination of several projections, Qin and Etesami \cite{qin2020randomized} devised a randomized accelerated projection algorithm, which has the faster convergence rate than the classic cyclic projection method.

In this paper, motivated by \cite{de2017sampling}, we first present a nonlinear sampling Kaczmarz-Motzkin (NSKM) method for a unconstrained problem \eqref{equ1} and establish the corresponding convergence theory with two different assumptions. Preliminary numerical results show that the NSKM method is more effective to solve the large-scale nonlinear equations than the NRK method in terms of calculation times. Then, inspired by the ideas in \cite{wang2013incremental}, \cite{qin2020randomized} and the NSKM method, for the system \eqref{equ1} with convex constraints we propose the projected sampling Kaczmarz-Motzkin (PSKM) method and the accelerated projected sampling Kaczmarz-Motzkin (APSKM) method. Applying the nonexpansive property of the projection and the local tangential cone condition of the system, we obtain the convergence analyses of the two new methods. The numerical results confirm the PSKM and APSKM methods have advantages over the PSGD method in terms of the number of iteration steps and calculation times .

The remaining part of the paper is organized as follows. In Section \ref{section2}, we present the NSKM method and analyzed its convergence. In Section \ref{section3}, we extend the NSKM method to get two variants of the NSKM method and prove the convergences of these methods. In Section \ref{section4}, we provide some numerical experiments to display the practical performence of the proposed methods. Finally, we finish this paper with a conclusion.

Throughout the paper, we use $\lvert\cdot\rvert$ to denote the scalar absolute value and  $\|\cdot\|$ to denote the vector 2-norm. The set of natural numbers is defined as $\mathbb{N}$. For a matrix $A\in\mathbb{R}^{m\times n}$, we use $\|A\|_F$, $\sigma_{min}(A)$ and $A(i,:)$ to denote the matrix Frobenius norm, the smallest non-zero singular value of matrix $A$ and the $i$th row of the matrix $A$, respectively. Let $P_C$ represent the metric projection onto the set $C$. We indicate by $E_{k-1}[\cdot]$ the expected value conditional on the first $k-1$th iterations, and from the law of the iterated expectation we have $E[E_{k-1}[\cdot]]=E[\cdot]$.
\section{The nonlinear sampling Kaczmarz-Motzkin method }
\label{section2}
In this section, we first present the NSKM method to solve the nonlinear equations
\begin{equation}\label{equation2}
	f(x)=0,
\end{equation} 
where $f:\mathscr{D}(f)\subseteq\mathbb{R}^n \to \mathbb{R}^m$ is a nolinear vector-valued function, $x\in \mathbb{R}^n$ is an unknown vector. The system \eqref{equation2} can also be written in the following form: 
\begin{equation*}
	f_i(x)=0, i=1,2,...,m,
\end{equation*}
where at least one $f_i:\mathscr{D}(f_i)\subseteq \mathbb{R}^n\to \mathbb{R}(i=1,2,...,m)$ are nonlinear operators. Then, the convergence analysis of the NSKM method is followed.
\subsection{The NSKM method}
In each iteration of the NRK method, we note that the NRK method needs to calculate $f(x)$, which is very expensive and inefficient when the size of nonlinear equations is large. Thus, we eager to design a more cheap algorithm at each update to avoid computing all $f(x)$. Motivated by the SKM method in \cite{de2017sampling}, we propose the NSKM method, in which, a sample of $\beta$ constraints is randomly selected from all entrys of $f(x)$ according to uniform probability and a portion of $f(x)$ need to be calculated instead of all $f(x)$. Moreover, in the NSKM method, we choose indicator $i_k$ by the maximal-residual criterion from the selected sample. The NSKM method can be formulated as follows.

\begin{algorithm}
	\caption{The nonlinear sampling Kaczmarz-Motzkin (NSKM) method }\label{alg2.1}
	\begin{algorithmic}[1]
		\Require $f(x)$, $\beta$, $x_0$, $k=1$, and maximum iteration steps $T$
		\Ensure $x_k$
		\While{iteration termination criterion does not hold and $k\leq T$ }
		\State Choose a sample of $\beta$ constraints, $\tau_k$, uniformly at random from all entrys of $f(x)$
		\State Compute residual of the selected sample $r_{\tau_k}=-f_{\tau_k}(x_{k-1})$
		\State  Set $i_{k}=\mathop{argmax}\limits_{i\in\tau_k}{\mid r_{i}\mid}$
		\State Compute gradient $g_{i_k}=\nabla f_{i_k}(x_{k-1})$
		\State Set $x_{k}=x_{k-1}+\frac{r_{i_k}}{\|g_{i_k}\|^2}g_{i_k}^T$
		\State $k=k+1$	 
		\EndWhile	
	\end{algorithmic}
\end{algorithm}
\subsection{Convergence analysis of the NSKM method}
In order to prove the convergence of Algorithm \ref{alg2.1}, we need to prepare some definitions, lemmas and corollaries at the top of this section.
\begin{definition}[\cite{wang2022nonlinear}]
	\label{definition1}
	A matrix $A\in \mathbb{R}^{m\times n}$ is called row bounded below, if there exists a positive number $\varepsilon$ such that $\|A(i,:)\|\geq \varepsilon$, for $1\leq i\leq m$.
\end{definition}
\begin{definition}[\cite{haltmeier2007kaczmarz}]
	\label{definition2}
	If there is a point $x_0\in\mathscr{D}(f)$ such that for every $i\in \{1,2,...,m\}$ and $\forall x_1,x_2\in \mathscr{B}_\rho (x_0)\subset \mathscr{D}(f)$ $ (\mathscr{B}_\rho(x_0)=\{x\lvert \|x-x_0\|\leq \rho\})$, there existis $\eta_i\in [0,\eta)$ $(\eta=\mathop{max}\limits_{i}{\eta_i}\textless \frac{1}{2})$ such that 
	\begin{equation}\label{def2.2.1}
		\lvert f_i(x_1)-f_i(x_2)-\nabla f_i(x_1)(x_1-x_2)\rvert \leq \eta_i\lvert f_i(x_1)-f_i(x_2)\rvert ,
	\end{equation}
	then the function $f:\mathscr{D}(f)\subset \mathbb{R}^n \rightarrow \mathbb{R}^m$ is referred to satisfy the local tagential cone condition in a ball $\mathscr{B}_\rho (x_0)$ of radius $\rho$ around $x_0$.
\end{definition}
\begin{lem}[\cite{wang2022nonlinear}]\label{lem2.1}
	Let $f(x)$ satisfy the local tangential cone condition in a ball $\mathscr{B}_\rho (x_0)$. Then, for $\forall x_1, x_2\in \mathscr{B}_\rho(x_0)\subset \mathscr{D}(f)$, we have 
	$$\mid f_i(x_1)-f_i(x_2)\mid \geq \frac{1}{1+\eta_i}\mid \nabla f_i(x_1)(x_1-x_2)\mid,\quad i\in \{1,2,...,m\}.$$
\end{lem}
By Lemma \ref{lem2.1}, the following corollary is naturally followed.
\begin{coroll}\label{coroll2.1}
	Let $f(x)$ satisfy the local tangential cone condition in a ball $\mathscr{B}_\rho (x_0)$. Then, for $\forall x_1, x_2\in \mathscr{B}_\rho(x_0)\subset \mathscr{D}(f)$, we have 
	$$\|f(x_1)-f(x_2)\|^2\geq \frac{1}{(1+\eta)^2}\| f'(x_1)(x_1-x_2)\|^2,$$
	where $\eta=\mathop{max}\limits_{i} {\eta_i} \textless \frac{1}{2}\quad(i=1,2,...,m).$
\end{coroll}
\begin{proof}
	\begin{align*}
		\|f(x_1)-f(x_2)\|^2
		&=\|(f_1(x_1)-f_1(x_2), f_2(x_1)-f_2(x_2), ..., f_m(x_1)-f_m(x_2))^T\|^2\\
		&=(f_1(x_1)-f_1(x_2))^2+(f_2(x_1)-f_2(x_2))^2+...+(f_m(x_1)-f_m(x_2))^2\\
		&\geq \frac{1}{(1+\eta_1)^2}(\nabla f_1(x_1)(x_1-x_2))^2+\frac{1}{(1+\eta_2)^2}(\nabla f_2(x_1)(x_1-x_2))^2+...\\
		&+\frac{1}{(1+\eta_m)^2}(\nabla f_m(x_1)(x_1-x_2))^2\\
		&\geq \frac{1}{(1+\eta)^2}\sum_{i=1}^{m}(\nabla f_i(x_1)(x_1-x_2))^2\\
		&=\frac{1}{(1+\eta)^2}\|f'(x_1)(x_1-x_2)\|^2,
	\end{align*}
	where the first inequality is obtained by Lemma \ref{lem2.1}.
\end{proof}
\begin{lem}\label{lem2.2}
	Let $f(x)$ satisfy the local tangential cone condition in a ball $\mathscr{B}_\rho (x_0)$ with $x^*\in \mathscr{B}_{\frac{\rho}{2}}(x_0)$. Then the sequence $\{x_k\}_{k=0}^\infty$ generated by Algorithm \ref{assumption2.1} is contained in $\mathscr{B}_\rho(x_0)\subset\mathscr{D}(f)$. Furthermore, we have
	\begin{equation}\label{lem2.2.1}
		\|x_{k+1}-x^*\|^2\leq\|x_{k}-x^*\|^2 -(1-2\eta_{i_{k+1}} )\frac{f_{i_{k+1}}^2(x_{k})}{\|\nabla f_{i_{k+1}}(x_{k})\|^2},
	\end{equation}
	where $f(x^*)=0.$
\end{lem}
\begin{proof}
	
	\begin{align*}
	    &\quad\|x_{k+1}-x^*\|^2-\|x_{k}-x^*\|^2\\
		&=\|x_{k+1}-x_{k}\|^2+2\left<x_{k+1}-x_{k},x_{k}-x^*\right>\\
		&=\|-\frac{f_{i_{k+1}}(x_{k})}{\|\nabla f_{i_{k+1}}(x_{k})\|^2}\nabla f_{i_{k+1}}(x_{k})^T\|^2+2\langle-\frac{f_{i_{k+1}}(x_{k})}{\|\nabla f_{i_{k+1}}(x_{k})\|^2}\nabla f_{i_{k+1}}(x_{k})^T, x_{k}-x^*\rangle\\
		&=\frac{f_{i_{k+1}}^2(x_{k})}{\|\nabla f_{i_{k+1}}(x_{k})\|^2}-2\frac{f_{i_{k+1}}(x_{k})}{\|\nabla f_{i_{k+1}}(x_{k})\|^2}\nabla f_{i_{k+1}}(x_{k})(x_{k}-x^*)\\
		&=\frac{f_{i_{k+1}}^2(x_{k})}{\|\nabla f_{i_{k+1}}(x_{k})\|^2}-2\frac{f_{i_{k+1}}(x_{k})}{\|\nabla f_{i_{k+1}}(x_{k})\|^2}f_{i_{k+1}}(x_{k})\\
		&+2\frac{f_{i_{k+1}}(x_{k})}{\|\nabla f_{i_{k+1}}(x_{k})\|^2}(f_{i_{k+1}}(x_{k})-f_{i_{k+1}}(x^*)-\nabla f_{i_{k+1}}(x_{k})(x_{k}-x^*)).			
	\end{align*}	
	
	When $k=0$, $x_0\in \mathscr{B}_\rho(x_0)$ and 	$\lvert f_i(x_0)-f_i(x^*)-\nabla f_i(x_0)(x_0-x^*)\rvert \leq \eta_i\lvert f_i(x_0)-f_i(x^*)\rvert$ $(i=1,2,...,m) $, then we have
	\begin{align}\label{lemma2.2.2}
			&\quad\|x_{1}-x^*\|^2-\|x_{0}-x^*\|^2\nonumber\\
			&=\frac{f_{i_{1}}^2(x_{0})}{\|\nabla f_{i_{1}}(x_{0})\|^2}+2\frac{f_{i_{1}}(x_{0})}{\|\nabla f_{i_{1}}(x_{0})\|^2}(f_{i_{1}}(x_{0})-f_{i_{1}}(x^*)-\nabla f_{i_{1}}(x_{0})(x_{0}-x^*))\nonumber\\
			&-2\frac{f_{i_{1}}(x_{0})}{\|\nabla f_{i_{1}}(x_{0})\|^2}f_{i_{1}}(x_{0})\nonumber \\
			&\leq\frac{f_{i_{1}}^2(x_{0})}{\|\nabla f_{i_{1}}(x_{0})\|^2}+2\eta_{i_{1}}\frac{\mid f_{i_{1}}(x_{0})\mid}{\|\nabla f_{i_{1}}(x_{0})\|^2}\mid f_{i_{1}}(x_{0})\mid -2\frac{f_{i_{1}}^2(x_{0})}{\|\nabla f_{i_{1}}(x_{0})\|^2}\nonumber \\
			&=-(1-2\eta_{i_{1}})\frac{f_{i_{1}}^2(x_{0})}{\|\nabla f_{i_{1}}(x_{0})\|^2}.
	\end{align}
	
	Since $x^*\in \mathscr{B}_{\rho/2}(x_0)$ and \eqref{lemma2.2.2}, we have
	$$\|x_1-x_0\|=\|x_1-x^*+x^*-x_0\|\leq \|x_1-x^*\|+\|x^*-x_0\|\leq \rho. $$
	
	Thus, $x_1\in \mathscr{B}_\rho(x_0)$. 
	
	We assume that when $k\leq n$ $(n\in \mathbb{N})$, $x_k\in \mathscr{B}_\rho(x_0)$ and \eqref{lem2.2.1} holds, then, for $k=n+1$, similar to the derivation of $k=0$, we have $x_{n+1}\in \mathscr{B}_\rho(x_0)$ and \eqref{lem2.2.1} holds.
\end{proof}
\begin{lem} [\cite{brinkhuis2020convex}]
	A proper convex function is a function $f:\mathscr{D}(f)\to \mathbb{R}$, where $\mathscr{D}(f) \subset \mathbb{R}^n$ is a nonempty convex set. Then Jensen's inequality
	$$f((1-\alpha)x+\alpha y)\leq (1-\alpha)f(x)+\alpha f(y), \forall \alpha \in [0,1], \forall x,y \in \mathscr{D}(f)$$ holds.
\end{lem}
The following Lemma \ref{lemm2.3.3} is very important in the next convergence analysis, so we review its proof in Appendix. 
\begin{lem} [\cite{brinkhuis2020convex}]\label{lemm2.3.3}
	Let an convex set $\mathscr{D}(f)\subset \mathbb{R}^n$ and a differentiable function $f:\mathscr{D}(f)\to \mathbb{R}$ be given. If f is convex, then $f(x)-f(y)\geq f'(y)(x-y)$ for all $x,y\in \mathscr{D}(f)$.
\end{lem}

\begin{lem}[\cite{de2017sampling}]\label{lem2.3}
	Suppose $\{a_i\}_{i=1}^n$ and $\{b_i\}_{i=1}^n$ are real sequences and $a_{i+1}\textgreater a_i\textgreater 0$ and $b_{i+1}\geq b_i\geq0$. Then
	\begin{equation*}
		\sum_{i=1}^{n}a_ib_i\geq \sum_{i=1}^{n}\bar{a}b_i,
	\end{equation*} 
	where $\bar{a}$ is the average $\bar{a}=\frac{1}{n}\sum_{i=1}^{n}a_i.$
\end{lem}
\begin{lem}[\cite{wang2022nonlinear}]\label{lem2.5}
	Let $a=\{a_1,a_2,...,a_n\}$ and $b=\{b_1,b_2,...,b_n\}$ be two arrays with real components and satisfy $a_j\geq 0,b_j\textgreater 0,j\in \{1,2,...,n\}$, then the following inequality is established
	$$\sum\limits_{j=1}^n \frac{a_j}{b_j}\geq \frac{\sum_{j=1}^n a_j}{\sum_{j=1}^n b_j}.$$
\end{lem}
\subsubsection{Convergence analysis \uppercase\expandafter{\romannumeral1}}
\begin{assumption}\label{assumption2.1}
	The following assumptions hold.
	\begin{itemize}
		\item[(i)] Nonlinear function $f:\mathscr{D}(f)\subseteq \mathbb{R}^n \rightarrow \mathbb{R}^m$ satisfies the local tangential cone condition in a ball $\mathscr{B}_{\rho}(x_0)$.
		\item[(ii)]For $\forall x \in \mathscr{D}(f)$, $f'(x)$ is row bounded below and full column rank matrix.
	\end{itemize}
\end{assumption}
\begin{thm}\label{the2.1}
	Assume that $f(x)$ satisfies Assumption \ref{assumption2.1} and $f(x)=0$ is solvable in $\mathscr{B}_{\frac{\rho}{2}}(x_0)$. Then the iteration sequence $\{x_k\}_{k=0}^{\infty}$ generated by the NSKM method converges to a solution $x^*\in \mathscr{B}_{\frac{\rho}{2}}(x_0)$ of $f(x)$ in expectation. Moreover, the mean squared iteration error satisfies 
	\begin{equation*}
		E\|x_k-x^*\|^2 \leq(1-\frac{(1-2\eta)\sigma_{min}^2(f'(x_{k-1}))}{m(1+\eta)^2\|f'(x_{k-1})\|_F^2})E\|x_{k-1}-x^*\|^2, k=1,2,...
	\end{equation*}
	where $\eta=\mathop{max}\limits_{i}{\eta_i}\textless \frac{1}{2}\quad(i=1,2,...,m).$
\end{thm}
\begin{proof}
	From Lemma \ref{lem2.2}, we have
	$$\|x_{k}-x^*\|^2\leq \|x_{k-1}-x^*\|^2- (1-2\eta_{i_k})\frac{f_{i_{k}}^2(x_{k-1})}{\|\nabla f_{i_{k}}(x_{k-1})\|^2}.$$
	
	By taking the conditional expectation on both sides of the above formula, we obtain
	$$E_{k-1} \|x_k-x^*\|^2 \leq\|x_{k-1}-x^*\|^2-E_{k-1}(1-2\eta_{i_k})\frac{f_{i_k}^2(x_{k-1})}{\|\nabla f_{i_k}(x_{k-1})\|^2}.$$
	
	Since  $\eta_{i_k}\in [0,\eta)$ $ (\eta=\mathop{max}\limits_{i}\eta_i\textless \frac{1}{2})$, we have that
	\begin{equation}\label{equation2.1.1}
		E_{k-1} \|x_k-x^*\|^2 
		\leq\|x_{k-1}-x^*\|^2-(1-2\eta)E_{k-1}\frac{f_{i_k}^2(x_{k-1})}{\|\nabla f_{i_k}(x_{k-1})\|^2}.
	\end{equation}
	
	Next, we consider $E_{k-1}\frac{f_{i_k}^2(x_{k-1})}{\|\nabla f_{i_k}(x_{k-1})\|^2}$. Let $f_{j}^2(x_{k-1})$ denote the $(j+\beta)th$ smallest entry of the $\{f_i^2(x_{k-1})\}_{i=1}^m$ (i.e., if we order all entries of $\{f_i^2(x_{k-1})\}_{i=1}^m$ from smallest to largest, $f_{j}^2(x_{k-1})$ is in the $(j+\beta)th$ position). Each sample has equal probability of being selected, $\binom{m}{\beta}^{-1}$. However, the selected frequency of each entry of $\{f_i^2(x_{k-1})\}_{i=1}^m$ depends on its size. For example, the largest entry of $\{f_i^2(x_{k-1})\}_{i=1}^m$ can be selected from all samples in which it appears, while the $\beta$th smallest entry of $\{f_i^2(x_{k-1})\}_{i=1}^m$ can be selected from only one sample. Therefore, we have that
	\begin{align}\label{equation2.1.2}
		E_{k-1}\frac{f_{i_k}^2(x_{k-1})}{\|\nabla f_{i_k}(x_{k-1})\|^2}&=\sum_{j=0}^{m-\beta}\frac{\binom{j+\beta-1}{\beta-1}}{\binom{m}{\beta}}\frac{f_{j}^2(x_{k-1})}{\|\nabla f_{j}(x_{k-1})\|^2}\notag \\
		&=\frac{1}{\binom{m}{\beta}}\sum_{j=0}^{m-\beta}\binom{j+\beta-1}{\beta-1}\frac{f_{j}^2(x_{k-1})}{\|\nabla f_{j}(x_{k-1})\|^2}\notag \\
		&\geq\frac{1}{\binom{m}{\beta}}\frac{\sum_{j=0}^{m-\beta} \binom{j+\beta-1}{\beta-1} f_{j}^2(x_{k-1})}{\sum_{j=0}^{m-\beta} \|\nabla f_{j}(x_{k-1})\|^2}\notag\\
		&\geq \frac{1}{\binom{m}{\beta}}\frac{\sum_{j=0}^{m-\beta} \sum_{l=0}^{m-\beta} \frac{\binom{l+\beta-1}{\beta-1}}{m-\beta+1}  f_{j}^2(x_{k-1})}{\|f'(x_{k-1})\|_F^2}\notag\\
		&=\sum_{j=0}^{m-\beta}\frac{1}{m-\beta+1}\frac{f_j^2(x_{k-1})}{\|f'(x_{k-1})\|_F^2}, 
	\end{align} 
	where the first inequality comes from Lemma \ref{lem2.5}, the second inequality is from Lemma \ref{lem2.3}, because $\{\binom{j+\beta-1}{\beta-1}\}_{j=0}^{m-\beta}$
	is strictly increasing and $f_j^2(x_{k-1})$ is non-decreasing, and the last equality follows from the fact that $\sum_{l=0}^{m-\beta}\binom{l+\beta-1}{\beta-1}=\binom{m}{\beta}$, which is known as the column-sum property of Pascal's triangle. 
	
	Let $s_{k-1}$ be the number of zero entries in $f(x_{k-1})$ and $V_{k-1}=max\{m-s_{k-1}, m-\beta+1\}$, then from \eqref{equation2.1.2} we can derive 
	\begin{align}\label{equation2.1.3}
		E_{k-1}\frac{f_{i_k}^2(x_{k-1})}{\|\nabla f_{i_k}(x_{k-1})\|^2}&\geq \frac{1}{m-\beta+1}min\left\{\frac{m-\beta+1}{m-s_{k-1}},1\right\}\frac{\|f(x_{k-1})\|^2}{\|f'(x_{k-1})\|_F^2}\notag\\
		&=\frac{1}{V_{k-1}}\frac{\|f(x_{k-1})\|^2}{\|f'(x_{k-1})\|_F^2}.
	\end{align}
	
	In accordance with formula \eqref{equation2.1.3}, the formula \eqref{equation2.1.1} then further results in the estimate
	\begin{align*}
		&\quad E_{k-1} \|x_k-x^*\|^2 \\
		&\leq\|x_{k-1}-x^*\|^2-(1-2\eta)	E_{k-1}\frac{f_{i_k}^2(x_{k-1})}{\|\nabla f_{i_k}(x_{k-1})\|^2}\\
		&\leq\|x_{k-1}-x^*\|^2-(1-2\eta)\frac{1}{V_{k-1}}\frac{\|f(x_{k-1})\|^2}{\|f'(x_{k-1})\|_F^2}\\
		&\leq\|x_{k-1}-x^*\|^2-\frac{1-2\eta}{m}\frac{\|f(x_{k-1})\|^2}{\|f'(x_{k-1})\|_F^2}\\ 
		&=\|x_{k-1}-x^*\|^2-\frac{1-2\eta}{m}\frac{\|f(x_{k-1})-f(x^*)\|^2}{\|f'(x_{k-1})\|_F^2}\\
		&\leq\|x_{k-1}-x^*\|^2-\frac{1-2\eta}{m\|f'(x_{k-1})\|_F^2}\frac{1}{(1+\eta)^2}\|f'(x_{k-1})(x_{k-1}-x^*) \|^2\\
		&\leq\|x_{k-1}-x^*\|^2-\frac{1-2\eta}{m\|f'(x_{k-1})\|_F^2}\frac{1}{(1+\eta)^2}\sigma_{min}^2(f'(x_{k-1}))\|x_{k-1}-x^*\|^2\\
		&=(1-\frac{(1-2\eta)\sigma_{min}^2(f'(x_{k-1}))}{m(1+\eta)^2\|f'(x_{k-1})\|_F^2})\|x_{k-1}-x^*\|^2,
	\end{align*}
	where the third inequality follows from the fact that $V_{k-1}\leq m$, the fourth inequality comes from Corollary \ref{coroll2.1}, the last inequality is from the fact that $\|f'(x_{k-1})(x_{k-1}-x^*)\|^2\geq\sigma_{min}^2(f'(x_{k-1}))\|(x_{k-1}-x^*)\|^2 $ when $f'(x_{k-1})$ is full column rank.
	
	By taking full expectation on the both sides, we can further obtain
	\begin{equation*}
		E\|x_k-x^*\|^2 \leq(1-\frac{(1-2\eta)\sigma_{min}^2(f'(x_{k-1}))}{m(1+\eta)^2\|f'(x_{k-1})\|_F^2})E\|x_{k-1}-x^*\|^2.
	\end{equation*}
	
	Since $\sigma_{min}^2(f'(x_{k-1}))\textless \|f'(x_{k-1})\|_F^2$, and $0\textless 1-2\eta \textless (1+\eta)^2$, we have 
	$$0\textless 1-\frac{(1-2\eta)\sigma_{min}^2(f'(x_{k-1}))}{m(1+\eta)^2\|f'(x_{k-1})\|_F^2}\textless 1.$$
	Therefore, the iteration sequence $\{x_k\}_{k=0}^\infty$ generated by the NSKM method converges to $x^*$ in expectation.
\end{proof}
\begin{remark}
	\item[(1)] To ensure that $\|\nabla f_{i_{k+1}} (x_k)\|\neq 0$ in $x_{k+1}=x_{k}- \frac{f_{i_{k+1}}(x_k)}{\|\nabla f_{i_{k+1}}(x_k)\|^2} \nabla f_{i_{k+1}}(x_k)^T$ $(k=0,1,2,...)$, we require that $f'(x)$ is row bounded below for every $x\in \mathscr{D}(f)$. 
	\item[(2)] For proving the desired convergence rate of the NSKM method, we need that $f'(x)$ is full column rank matrix for every $ x\in \mathscr{D}(f)$.
\end{remark}
\subsubsection{Convergence analysis \uppercase\expandafter{\romannumeral2}}
\begin{assumption}\label{assum2.2}
	The following assumptions hold.
	\begin{itemize}
		\item[(i)] $f_i:\mathscr{D}(f)\subseteq \mathbb{R}^n \rightarrow \mathbb{R}$ $(i=1,2,...,m)$ on the nonempty convex set $\mathscr{D}(f)$ are convex functions.
		\item[(ii)]For $\forall x \in \mathscr{D}(f), f'(x)$ is row bounded below.
		\item[(iii)] $f_i(x)\geq 0$ $(i=1,2,...,m)$ for $\forall x\in \mathscr{D}(f).$
	\end{itemize}
\end{assumption}
\begin{thm}\label{the2.2}
	Suppose that Assumption \ref{assum2.2} holds and there exists $x^*$ such that $f(x^*)=0$. Then the iteration sequence $\{x_k\}_{k=0}^{\infty}$ generated by the NSKM method converges to a solution $x^*$ of $f(x)$. Moreover, iteration error satisfies 
	\begin{equation*}
		\|x_k-x^*\|^2 \leq\|x_{k-1}-x^*\|^2-\frac{f_{i_k}^2(x_{k-1})}{\|\nabla f_{i_k}(x_{k-1})\|^2}, k=1,2,....
	\end{equation*}
\end{thm}
\begin{proof}
	\begin{align*}
		&\quad\|x_{k}-x^*\|^2-\|x_{k-1}-x^*\|^2\\
		&=\frac{f_{i_{k}}^2(x_{k-1})}{\|\nabla f_{i_{k}}(x_{k-1})\|^2}-2\frac{f_{i_{k}}(x_{k-1})}{\|\nabla f_{i_{k}}(x_{k-1})\|^2}\nabla f_{i_{k}}(x_{k-1})(x_{k-1}-x^*)\\
		&=\frac{f_{i_{k}}^2(x_{k-1})}{\|\nabla f_{i_{k}}(x_{k-1})\|^2}+2\frac{f_{i_{k}}(x_{k-1})}{\|\nabla f_{i_{k}}(x_{k-1})\|^2}(f_{i_{k}}(x_{k-1})+\nabla f_{i_{k}}(x_{k-1})(x^*-x_{k-1}))\\
		&-2\frac{f_{i_{k}}(x_{k-1})}{\|\nabla f_{i_{k}}(x_{k-1})\|^2}f_{i_{k}}(x_{k-1})\\
		&\leq\frac{f_{i_{k}}^2(x_{k-1})}{\|\nabla f_{i_{k}}(x_{k-1})\|^2}-2\frac{f_{i_{k}}^2(x_{k-1})}{\|\nabla f_{i_{k}}(x_{k-1})\|^2}\\
		&=-\frac{f_{i_{k}}^2(x_{k-1})}{\|\nabla f_{i_{k}}(x_{k-1})\|^2},		
	\end{align*}	
	where the first equality is obtained by the proof of Lemma \ref{lem2.2} and the first inequality comes from Lemma \ref{lemm2.3.3} and (iii) of Assumption \ref{assum2.2}.
\end{proof}
\begin{remark}\label{remark2.2}
	\item[(1)]When $f_i(x)$ is a convex function or satisfies \eqref{def2.2.1} in a ball $\mathscr{B}_{\rho}(x_0)$ for $\forall i\in\{1,2,...,m\}$, by Theorem \ref{the2.1} and Theorem \ref{the2.2}, we can get that the iteration sequence generated by the NSKM method also converges to a solution of \eqref{equation2}. The result be stated in the following Corollary \ref{corollary2.2}.
	\item[(2)]When $f(x)$ satisfies Assumption \ref{assum2.2} or the conditions of Corollary \ref{corollary2.2}, the convergences of the NRK method and the NURK method can also be guaranteed. 
\end{remark}
\begin{coroll}\label{corollary2.2}
	Suppose that $f_i(x):\mathscr{D}(f_i)\to \mathbb{R}$ satisfies \eqref{def2.2.1} in a ball $\mathscr{B}_{\rho}(x_0)$ or is a nonnegative convex function. Let $f'(x)$ is row bounded below for $\forall x\in \mathscr{D}(f)$ and $f(x)=0$ is solvable in $\mathscr{B}_{\frac{\rho}{2}}(x_0)$. Then the iteration sequence $\{x_k\}_{k=0}^{\infty}$ generated by the NSKM method converges to a solution $x^*\in \mathscr{B}_{\frac{\rho}{2}}(x_0)$ of $f(x)$.
\end{coroll}
\section{The variants of the NSKM method}
\label{section3}
In this section, we will provide two variants of the NSKM method that converge to a solution of \eqref{equ1}. We note that a projection onto the closed convex sets has one famous nonexpansive property, inspired by which, we present the PSKM method. Furthermore, in \cite{qin2020randomized}, Qin and Etesami presented a randomized accelerated projected algorithm which had  the faster convergence rate compared with the alternating projected method. Thus, we devise the APSKM method.
\subsection{The PSKM method}
The PSKM method can be splited into two computational stages. The first stage uses the one step NSKM update for the nonlinear equation $f(x)=0$ and gets $x_{k-\frac{1}{2}}$, and the second stage outputs $x_k$ by utilizing one step randomized projection onto the finite nonempty closed convex sets $C_i$ $(i=1,2,...,k_c)$. The pseudo-code of the PSKM method is given in the following Algorithm \ref{alg3.1.1}.
\begin{algorithm}
	\caption{The projected sampling Kaczmarz-Motzkin (PSKM) method }
	\label{alg3.1.1}
	\begin{algorithmic}[1]
		\Require $f(x)$, $\beta$, $x_0$, $k=1$, number of the nonempty closed convex sets $k_c$ and maximum iteration steps $T$
		\Ensure $x_k$.
		\While{iteration termination criterion does not hold and $k\leq T$}  
		\State Choose a sample of $\beta$ constraints, $\tau_k$, uniformly at random from all entrys of $f(x)$
		\State Compute residual of the selected sample $r_{\tau_k}=-f_{\tau_k}(x_{k-1})$
		\State Set $i_{k}=\mathop{argmax}\limits_{i\in\tau_k}{\mid r_{i}\mid}$
		\State Compute gradient $g_{i_k}=\nabla f_{i_k}(x_{k-1})$.
		\State Set $x_{k-\frac{1}{2}}=x_{k-1}+ \frac{r_{i_k}}{\|g_{i_k}\|^2}g_{i_k}^T$
		\State Choose a indicator, $\alpha_k$, uniformly at random from the set $\{1,2,...,k_c\}$
		\State$x_k=P_{C_{\alpha_k}}(x_{k-\frac{1}{2}})$
		\State k=k+1
		\EndWhile		
	\end{algorithmic}
\end{algorithm}
\subsubsection{Convergence analysis of the PSKM method}
For proving the convergence of the PSKM method, we will first describe a crucial lemma.
\begin{lem}[\cite{qin2020randomized}]\label{lem3.1}
	Given a nonempty closed convex set $X\subset \mathbb{R}^d$ and a point $y\in \mathbb{R}^d$, the following relations hold for the projection $P_X(y)$ of the point $y$ on $X$ and for all $x\in X$,
	$$\langle P_X(y)-y, x-P_X(y) \rangle \geq 0,$$
	$$\|P_X(x)-P_X(y)\|^2\leq \|x-y\|^2-\|y-P_X(y)\|^2,$$
	with equality in the second relation if and only if $x=P_X(y)$ or $y\in X$.
\end{lem}
\begin{thm}\label{theo3.1}
	Suppose that $f(x)$ satisfies Assumption \ref{assumption2.1} and \eqref{equ1} is solvable in $\mathscr{B}_{\frac{\rho}{2}}(x_0)$. Then the iteration sequence $\{x_k\}_{k=0}^{\infty}$ generated by the PSKM method converges to a solution $x^*\in \mathscr{B}_{\frac{\rho}{2}}(x_0)$ of \eqref{equ1} in expectation. Moreover, the mean squared iteration error satisfies 
	\begin{equation*}
		E\|x_k-x^*\|^2 \leq(1-\frac{(1-2\eta)\sigma_{min}^2(f'(x_{k-1}))}{m(1+\eta)^2\|f'(x_{k-1})\|_F^2})E\|x_{k-1}-x^*\|^2, k=1,2,... 
	\end{equation*}
	where $\eta=\mathop{max}\limits_i{\eta_i}\textless \frac{1}{2}\quad(i=1,2,...,m).$
\end{thm}
\begin{proof}
	By Lemma \ref{lem3.1}, we have that
	$$\|x_{k}-x^*\|^2=\|P_{C_{\alpha_{k}}}(x_{k-\frac{1}{2}})-P_{C_{\alpha_{k}}}(x^*)\|^2\leq\|x_{k-\frac{1}{2}}-x^*\|^2.$$
	Since $x_{k-\frac{1}{2}}$ is generated by the NSKM method, from Theorem \ref{the2.1}, we can get 
	$$E\|x_{k-\frac{1}{2}}-x^*\|^2 \leq(1-\frac{(1-2\eta)\sigma_{min}^2(f'(x_{k-1}))}{m(1+\eta)^2\|f'(x_{k-1})\|_F^2})E\|x_{k-1}-x^*\|^2.$$
	Thus,\begin{align*}
		E\|x_{k}-x^*\|^2 &\leq E\|x_{k-\frac{1}{2}}-x^*\|^2\\ &\leq(1-\frac{(1-2\eta)\sigma_{min}^2(f'(x_{k-1}))}{m(1+\eta)^2\|f'(x_{k-1})\|_F^2})E\|x_{k-1}-x^*\|^2.
	\end{align*}
	By deducing in Theorem \ref{the2.1}, we have
	$$0\textless1-\frac{(1-2\eta)\sigma_{min}^2(f'(x_{k-1}))}{m(1+\eta)^2\|f'(x_{k-1})\|_F^2}\textless 1.$$
	Therefore, the iteration sequence generated by the PSKM method converges to a solution $x^*\in \mathscr{B}_{\frac{\rho}{2}}(x_0)$ of \eqref{equ1} in expectation.
\end{proof}
Similar to the proof of Theorem \ref{theo3.1}, we can get the following two theorems.
\begin{thm}
	Suppose that $f(x)$ satisfies Assumption \ref{assum2.2} and there exists $x^*$ such that $f(x^*)=0$, $x^*\in C$. Then the iteration sequence $\{x_k\}_{k=0}^{\infty}$ generated by the PSKM method converges to a solution $x^*$ of \eqref{equ1}. Moreover, iteration error satisfies 
	\begin{equation*}
		\|x_k-x^*\|^2 \leq\|x_{k-1}-x^*\|^2-\frac{f_{i_k}^2(x_{k-1})}{\|\nabla f_{i_k}(x_{k-1})\|^2}, k=1,2,...
	\end{equation*}
\end{thm}
\begin{thm}\label{thm3.3}
	Assume that $f_i(x):\mathscr{D}(f_i)\to\mathbb{R}$ satisfies \eqref{def2.2.1} in a ball $\mathscr{B}_{\rho}(x_0)$ or is a nonnegative convex function. Let $f'(x)$ is row bounded below for $\forall x\in \mathscr{D}(f)$ and \eqref{equ1} is solvable in $\mathscr{B}_{\frac{\rho}{2}}(x_0)$. Then the iteration sequence $\{x_k\}_{k=0}^{\infty}$ generated by the PSKM method converges to a solution $x^*\in \mathscr{B}_{\frac{\rho}{2}}(x_0)$ of \eqref{equ1}.
\end{thm}
\subsection{The APSKM method}
The APSKM method also contains two computational stages. The first stage is similar to that of the PSKM method. However, the second stage is vastly different, which is divided into two cases to get $x_k$. We give the APSKM method in Algorithm \ref{alg3.1}.
\begin{algorithm}
	\caption{The accelerated projected sampling Kaczmarz-Motzkin (APSKM) method }
	\label{alg3.1}
	\begin{algorithmic}[1]
		\Require $f(x)$, $\beta$, $x_0$, $k=1$, $\delta$, number of the closed convex sets $k_c$ and maximum iteration steps $T$
		\Ensure $x_k$
		\While{iteration termination criterion does not hold and $k\leq T$}
		\State Choose a sample of $\beta$ constraints, $\tau_k$, uniformly at random from all entrys of $f(x)$
		\State Compute residual of the selected sample $r_{\tau_k}=-f_{\tau_k}(x_{k-1})$
		\State Set $i_{k}=\mathop{argmax}\limits_{i\in\tau_k}{\mid r_{i}\mid}$
		\State Compute gradient $g_{i_k}=\nabla f_{i_k}(x_{k-1})$
		\State Set $x_{k-\frac{4}{5}}=x_{k-1}+ \frac{r_{i_k}}{\|g_{i_k}\|^2}g_{i_k}^T$
		\State Choose two indicators, $\alpha_1^k$, $\alpha_2^k$, uniformly at random from the set $\{1,2,...,k_c\}$
		\State$x_{k-\frac{3}{5}}=P_{C_{\alpha_1^k}}(x_{k-\frac{4}{5}})$
		\State$x_{k-\frac{2}{5}}=P_{C_{\alpha_2^k}}(x_{k-\frac{3}{5}})$
		\If{$\|x_{k-\frac{2}{5}}-x_{k-\frac{3}{5}}\|_\infty\textless \delta$} 
		\State $x_k=x_{k-\frac{2}{5}}$
		\Else
		\State$x_{k-\frac{1}{5}}=P_{C_{\alpha_1^k}}(x_{k-\frac{2}{5}})$
		\State Calculate $\lambda_k=\frac{\|x_{k-\frac{3}{5}}-x_{k-\frac{2}{5}}\|^2}{\left<x_{k-\frac{3}{5}}-x_{k-\frac{1}{5}}, x_{k-\frac{3}{5}}-x_{k-\frac{2}{5}}\right>}$\\
		\State $x_{k}=x_{k-\frac{3}{5}}+\lambda_k(x_{k-\frac{1}{5}}-x_{k-\frac{3}{5}})$
		\EndIf
		\State k=k+1
		\EndWhile
	\end{algorithmic}
\end{algorithm}
\subsubsection{Convergence analysis of the APSKM method} 
We will provide a lemma to prove the convergence of the APSKM method.
\begin{lem}[\cite{qin2020randomized}]\label{lem3.2}
	Given two closed convex sets $C_{\alpha_1^k}$ and $C_{\alpha_2^k}$, with $C_{\alpha_1^k}\cap C_{\alpha_2^k}\neq \emptyset$. The parameter $\lambda_k$ and the points $x_{k-\frac{3}{5}}, x_{k-\frac{2}{5}}, x_{k-\frac{1}{5}}, x_{k}$ are generated according to Step$14$, Step$8$, Step$9$, Step$13$, Step$16$ of Algorithm \ref{alg3.1}, respectively. Then we have the following property:\\
	for $\forall x \in C_{\alpha_1^k}\cap C_{\alpha_2^k},$
	\begin{align}\label{lemma3.2.1}
		\|x_{k}-x\|^2&\leq\|x_{k-\frac{4}{5}}-x\|^2.
	\end{align}
\end{lem}
\begin{proof}
	See Appendix.
\end{proof}
\begin{thm}\label{the3.3}
	Suppose that $f(x)$ satisfies Assumption \ref{assumption2.1} and \eqref{equ1} is solvable in $\mathscr{B}_{\frac{\rho}{2}}(x_0)$. Then the iteration sequence $\{x_k\}_{k=0}^{\infty}$ generated by the APSKM method converges to a solution $x^*\in \mathscr{B}_{\frac{\rho}{2}}(x_0)$ of \eqref{equ1} in expectation. Moreover, the mean squared iteration error satisfies 
	\begin{equation*}
		E\|x_k-x^*\|^2 \leq(1-\frac{(1-2\eta)\sigma_{min}^2(f'(x_{k-1}))}{m(1+\eta)^2\|f'(x_{k-1})\|_F^2})E\|x_{k-1}-x^*\|^2, k=1,2,... 
	\end{equation*}
	where $\eta=\mathop{max}\limits_i{\eta_i}\textless \frac{1}{2}\quad(i=1,2,...,m).$
\end{thm}
\begin{proof}
	If $\|x_{k-\frac{2}{5}}-x_{k-\frac{3}{5}}\|_\infty\textless \delta$, by Lemma \ref{lem3.1}, we have
	$$\|x_{k}-x^*\|^2\leq\|x_{k-\frac{4}{5}}-x^*\|^2.$$
	If $x_k$ is obtained by Step$16$ of Algorithm \ref{alg3.1}, from Lemma \ref{lem3.2}, we obtain that
	$$\|x_{k}-x^*\|^2\leq\|x_{k-\frac{4}{5}}-x^*\|^2.$$
	Since $x_{k-\frac{4}{5}}$ is obtained by one step update of the NSKM method, from Theorem \ref{the2.1}, we can get
	$$E\|x_{k-\frac{4}{5}}-x^*\|^2 \leq(1-\frac{(1-2\eta)\sigma_{min}^2(f'(x_{k-1}))}{m(1+\eta)^2\|f'(x_{k-1})\|_F^2})E\|x_{k-1}-x^*\|^2.$$
	Thus,
	$$E\|x_k-x^*\|^2\leq E\|x_{k-\frac{4}{5}}-x^*\|^2 \leq(1-\frac{(1-2\eta)\sigma_{min}^2(f'(x_{k-1}))}{m(1+\eta)^2\|f'(x_{k-1})\|_F^2})E\|x_{k-1}-x^*\|^2.$$
	By the derivation of Theorem \ref{the2.1}, we have that $$0\textless1-\frac{(1-2\eta)\sigma_{min}^2(f'(x_{k-1}))}{m(1+\eta)^2\|f'(x_{k-1})\|_F^2}\textless 1.$$
	Thus, the iteration sequence $\{x_k\}_{k=0}^{\infty}$ generated by the APSKM method converges to a solution $x^*\in \mathscr{B}_{\frac{\rho}{2}}(x_0)$ of \eqref{equ1} in expectation.
\end{proof}
Similar to the proof of Theorem \ref{the3.3}, we can get the following two theorems.
\begin{thm}
	Suppose that $f(x)$ satisfies Assumption \ref{assum2.2} and there exists $x^*$ such that $f(x^*)=0$, $x^*\in C$. Then the iteration sequence $\{x_k\}_{k=0}^{\infty}$ generated by the APSKM method converges to a solution $x^*$ of \eqref{equ1}. Moreover, iteration error satisfies 
	\begin{equation*}
		\|x_k-x^*\|^2 \leq\|x_{k-1}-x^*\|^2-\frac{f_{i_k}^2(x_{k-1})}{\|\nabla f_{i_k}(x_{k-1})\|^2}, k=1,2,...
	\end{equation*}
\end{thm}
\begin{thm}
	Assume that $f_i(x):\mathscr{D}(f_i)\to\mathbb{R}$  satisfies \eqref{def2.2.1} in a ball $\mathscr{B}_{\rho}(x_0)$ or is a nonnegative convex function. Let $f'(x)$ is row bounded below for $\forall x\in \mathscr{D}(f)$ and \eqref{equ1} is solvable in $\mathscr{B}_{\frac{\rho}{2}}(x_0)$. Then the iteration sequence $\{x_k\}_{k=0}^{\infty}$ generated by the APSKM method converges to a solution $x^*\in \mathscr{B}_{\frac{\rho}{2}}(x_0)$ of \eqref{equ1}.
\end{thm}
\section{Numerical experiments}
\label{section4}
In this section, firstly, on some given large-scale nonlinear equations, we run the NSKM method while varying the sample size, $\beta$ and compare the NSKM method with the NRK method. Secondly, we investigate the performances of the PSKM and APSKM methods and compare them with the PSGD method by testing some large-scale constrianed nonlinear equations. In our implementations, the stopping criterion is  
$$RSE=\frac{\|x_k-x^*\|^2}{\|x^*\|^2}\leq \varepsilon,$$
or the maximum iteration steps 500000 being reached. If the number of iteration steps exceeds 500000, it is denoted as "-". IT and CPU denote the number of iteration steps and CPU times (in seconds), respectively. IT and CPU are the medians of the required iterations steps and the elapsed CPU times with respect to 10 times repeated runs of the corresponding methods. All experiments are carried out using MATLAB (version R2020b) on a laptop with 2.20-GHZ intel Core i7-10870H processor, 16 GB memory, and Windows 10 operating system.
\subsection{Experiments on some large-scale nonlinear equations} 
\textbf{Problem 1} The nonlinear equation is taken as $$f_i(x)=(e^{x_i-1}-1)^2,i=1,2,...,m.$$ Obviously, the solution of the problem is $x^*=ones(n,1)$ and the problem satisfies (i), (iii) of Assumption \ref{assum2.2}. For (ii), it is possible that there exists $1\leq j\leq m$ such that $\|\nabla f_j(x_k)\|^2=0$, but if the iteration don't terminate, from the construction of Algorithm \ref{alg2.1}, we can known that the probability of $\|\nabla f_j(x_k)\|=0$ is very small when $\beta$ has the right size. Similarly, the NRK method can also be used to solve the problem. However, the nonlinear uniformly randomized (NURK) method \cite{wang2022nonlinear} don't guarantee to solve the problem successfully. In our experiments, we set $m=5000$, the initial value $x_0=0.5*ones(n,1)$. 

\textbf{Problem 2} The problem can be viewed as a modification of Chained Powell singlar function in \cite{lukvsan2018problems}. 
\begin{align*}
	f_k(x)&=x_i+10x_{i+1}-11, \quad mod(k,4)=1,\\
	f_k(x)&=\sqrt{5}(x_{i+2}-x_{i+3}),\quad mod(k,4)=2,\\
	f_k(x)&=(x_{i+1}-2x_{i+2}+1)^2,\quad mod(k,3)=3,\\
	f_k(x)&=\sqrt{10}(x_i-x_{i+3})^2,\quad mod(k,4)=0,\\
	m&=2(n-2),\quad i=2div(k+3,4)-1,
\end{align*} 
where div$(\cdot)$ is the division operation. The function satisfies conditions of Corollary \ref{corollary2.2} and for (ii) of Assumption \ref{assum2.2}, as stated in Problem 1. Thus, the NSKM method and the NRK method can be utilized to solve the problem. The problem has a solution $x^*=ones(n,1)$. In our work, we set $n=5000$, $m=2(n-2)=9996$, the initial value $x_0=0.5*ones(n,1)$.

The numerical results for Problem 1 and Problem 2 are shown in Fig. \ref{fig1}. Note that when $\beta=1$ the NSKM method is the NURK method. From Fig. \ref{fig1}, we can find that the curves of the NSKM method with $\beta=10,50,100$ are decreasing much more quickly than that of the other methods with respect to the increase of CPU times, which illustrates that the NSKM method with the sample of the right size performs better than the NRK and NURK methods in terms of CPU times.
\begin{figure*}
	\includegraphics[width=0.5\textwidth]{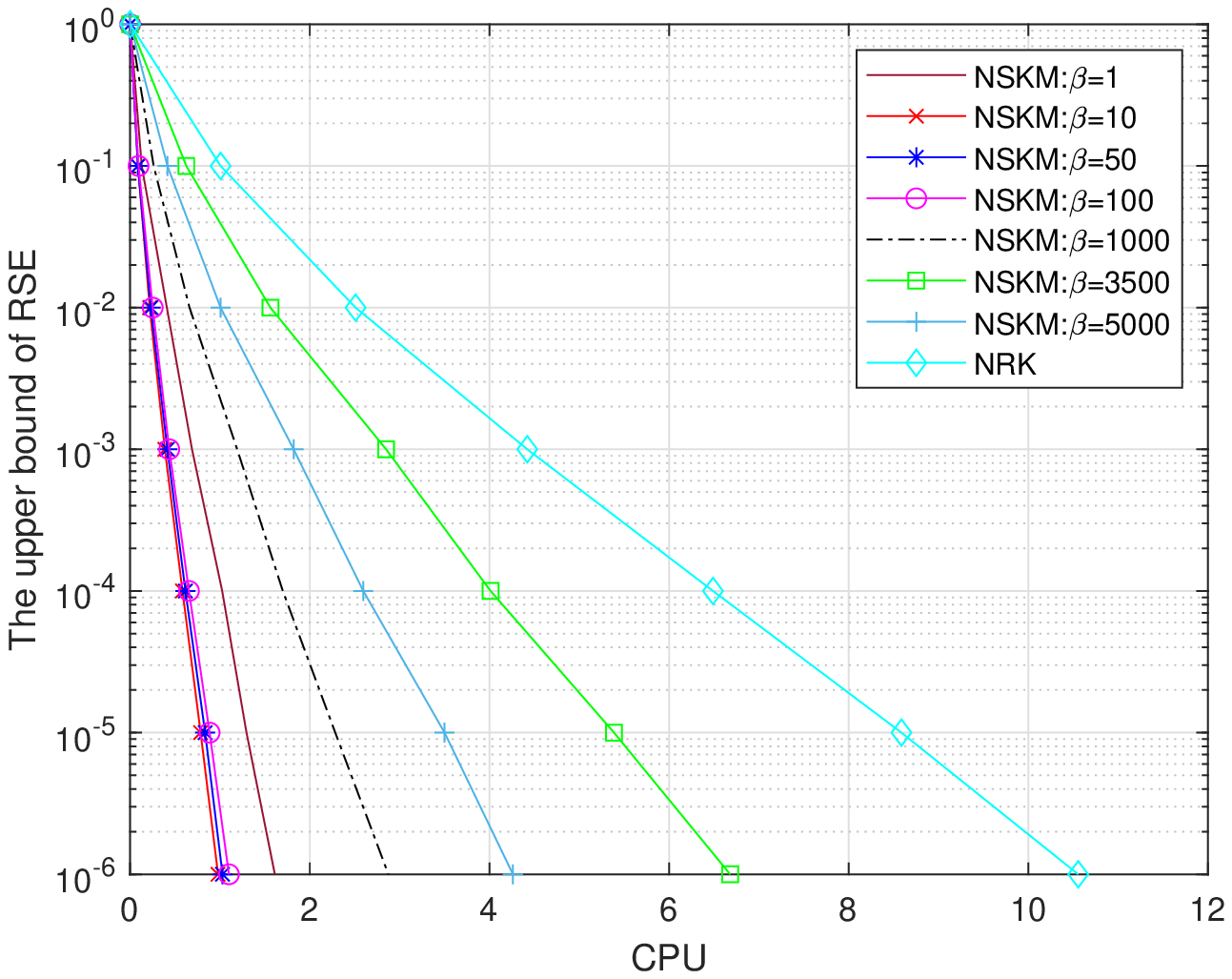}
	\includegraphics[width=0.5\textwidth]{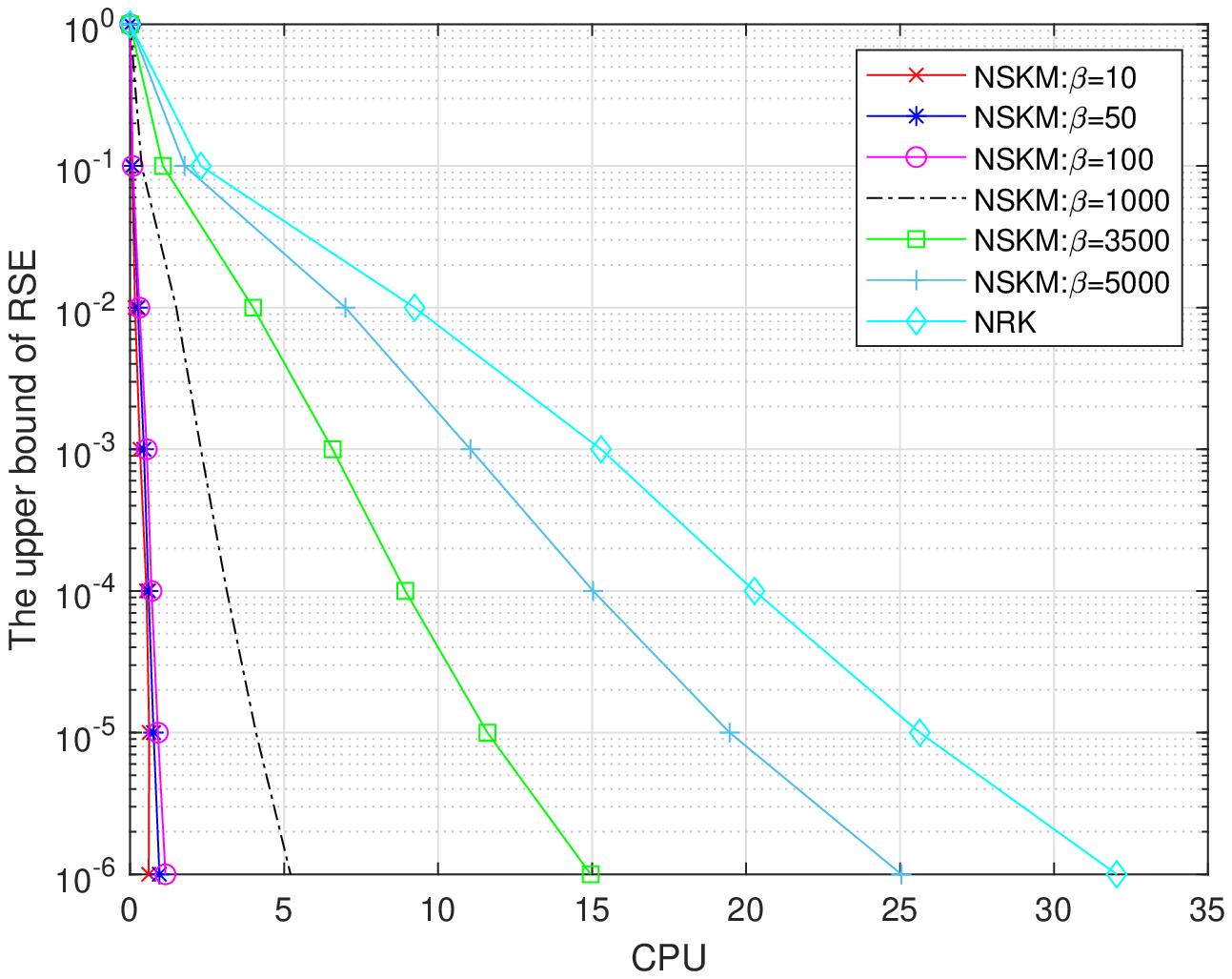}
	\caption{Pictures of the upper bound of the RSE versus CPU for the NRK method and the NSKM method with different sample for Problem 1 (left) and Problem 2 (right)}
	\label{fig1} 
\end{figure*}  
\subsection{Experiments on some large-scale nonlinear equations with finite convex constraints}
To examine the convergence of the PSKM and APSKM methods, we use Problem 1 and Problem 2 with $C=\{x:Ax = b\}$ or $\{x:Ax \leq b\}$. The projector $P_{C_i}$ onto the nonempty closed set $C_i$ can be defined as follows.
\begin{itemize}
	\item[1)]  When $C=\{x:Ax \leq b\}$, in \cite{bauschke1996projection}, $P_{C_i}x=x-\frac{(\langle a_i,x\rangle-b_i)^+}{\|a_i\|^2}a_i^T$, where $(\langle a_i,x\rangle-b_i)^+=max\{\langle a_i,x\rangle-b_i,0\}$.
	\item[2)] When $C=\{x:Ax = b\}$, $P_{C_i}x=x+\frac{b_i-\langle a_i,x\rangle}{\|a_i\|^2}a_i^T.$
\end{itemize}
Thus, the constrained nonlinear equations are determined.

In all simulations, we choose $\beta=50$, the initial guess $x_0=0.5*ones(n,1)$ and $\delta=10^{-10}$. In the PSGD algorithm, we
select a fixed step size, which is the best experimental result by trial and error. In the first simulation, $A$ is randomly generated from the Guassian distribution, $b=Ax^*$ for equality constraints and $b=Ax^*+abs(\delta1)$ for inequality constraints where $\delta1$ is generated by the MATLAB function \textbf{randn}. The results are shown in Table \ref{table1} and Table \ref{table2}, from which, we can clearly observe that the PSKM method requires less computing time than the PSGD and APSKM methods. In the second simulation, we consider $C=\{x:\langle A, x\rangle = b\}$ with coefficient matrix $A\in \mathbb{R}^{k_c \times n}$ on [$\xi$,1], which is generated from the MATLAB function \textbf{rand}. Note that the closer $\xi$ tends to 1, the stronger the correlation of $A$ is. Table \ref{table3} and Table \ref{table4} show that the APSKM method has a significantly better performance than other methods. Moreover, iteration counts and CPU times are decreasing with respect to the increase of $\xi$, which reveals that the stronger the correlation of $A$ is, the better performence the APSKM method has.
\begin{table}[!htbp]
	\centering\small
	\caption{IT and CPU of the PSGD, PSKM and APSKM methods for Problem 1 with $C=\{x:\langle A, x\rangle \leq b\}$ and $\varepsilon=10^{-3}$}
	\label{table1}
	\scalebox{0.83}{
	\begin{tabular}{cccccccc}
		\toprule  
		\multicolumn{3}{c}{ $m$ }   &3000&5000 &7000&9000\\
		\hline
		\multirow{6}*{$k_c=300$}&\multirow{2}*{PSGD}  & IT  &101567&98220&125724&220820\\
		& & CPU       &5.2485 &6.6294 &10.9787 &25.8523 \\
		\cline{2-7}
		&\multirow{2}*{PSKM}  & IT  &\textbf{8832}&\textbf{15438}&22055&28512\\
		& & CPU       &\textbf{0.5189} &\textbf{1.2335} &\textbf{2.1537} &\textbf{3.6497} \\
		\cline{2-7}
		&\multirow{2}*{APSKM}  & IT         &8855&15459&\textbf{22024}&\textbf{28471}\\
		& &CPU        &0.9579 &2.1683 &4.1890 &6.5978 \\ 
		\midrule
		\multirow{6}*{$k_c=500$}&\multirow{2}*{PSGD}  & IT  &67795&195818&162349&198686\\
		& & CPU       &4.6462 &16.4433 &18.5343 &29.8744 \\
		\cline{2-7}
		&\multirow{2}*{PSKM}  & IT  &\textbf{8765}&\textbf{15042}&\textbf{21599}&28159\\
		& & CPU       &\textbf{0.6111} &\textbf{1.5549} &\textbf{2.8039} &\textbf{4.6362} \\
		\cline{2-7}
		&\multirow{2}*{APSKM}  & IT         &8882&15092&21604&\textbf{28144}\\ 
		& &CPU        &1.2722 &2.5799 &4.5439 &8.7853 \\
		\midrule
		\multirow{6}*{$k_c=1000$}&\multirow{2}*{PSGD}  & IT  &87630&118458&238101&243719\\
		& & CPU       &7.6124 &11.2738 &24.3570 &30.5528 \\
		\cline{2-7}
		&\multirow{2}*{PSKM}  & IT  &\textbf{7841}&\textbf{14140}&\textbf{20611}&\textbf{27055}\\
		& & CPU       &\textbf{0.7259} &\textbf{1.4688} &\textbf{2.3765} &\textbf{3.6959} \\
		\cline{2-7}
		&\multirow{2}*{APSKM}  & IT         &7942&14338&20811&27165\\ 
		& &CPU        &1.4613 &3.0045 &4.7686 &7.0546 \\
		\bottomrule 
	\end{tabular}}
\end{table}
\begin{table}[!htbp]
	\centering\small
	\caption{IT and CPU of the PSGD, PSKM and APSKM methods for Problem 2 with $C=\{x:\langle A, x\rangle = b\}$ and $\varepsilon=10^{-3}$}
	\label{table2}
	\scalebox{0.83}{
	\begin{tabular}{cccccccc}
		\toprule  
		\multicolumn{3}{c}{ $m$ }   &3000&5000 &7000&9000\\
		\hline
		\multirow{6}*{$k_c=300$}&\multirow{2}*{PSGD}  & IT  &37830&71127&96602&130768\\
		& & CPU       &0.8371 &2.9615 &4.6478 &7.5615 \\
		\cline{2-7}
		&\multirow{2}*{PSKM}  & IT  &\textbf{4545}&\textbf{7829}&\textbf{11179}&\textbf{14186}\\
		& & CPU       &\textbf{0.1401 }&\textbf{0.3957} &\textbf{0.6233} &\textbf{1.0056} \\
		\cline{2-7}
		&\multirow{2}*{APSKM}  & IT         &4580&7971&11343&14279\\
		& &CPU        &0.2518 &0.8005 &1.3630 &2.1831 \\
		\midrule
		\multirow{6}*{$k_c=500$}&\multirow{2}*{PSGD}  & IT  &33483&65978&90915&119675\\
		& & CPU       &0.9642 &3.4890 &5.1170 &8.8733 \\
		\cline{2-7}
		&\multirow{2}*{PSKM}  & IT  &\textbf{4066}&\textbf{7277}&\textbf{11009}&\textbf{13969}\\
		& & CPU       &\textbf{0.1459} &\textbf{0.4430 }&\textbf{0.7782} &\textbf{1.2064} \\
		\cline{2-7}
		&\multirow{2}*{APSKM}  & IT         &4096&7435&11151&14089\\
		& &CPU        &0.2586 &1.0215 &1.8487 &2.6214 \\
		\midrule
		\multirow{6}*{$k_c=1000$}&\multirow{2}*{PSGD}  & IT  &19537&52287&80825&90823\\
		& & CPU       &0.6997 &3.3525 &4.7413 &6.3426 \\
		\cline{2-7}
		&\multirow{2}*{PSKM}  & IT  &3671&6614&9811&\textbf{13178}\\
		& & CPU       &\textbf{0.1673} &\textbf{0.4941 }&\textbf{0.6794} &\textbf{1.1053} \\
		\cline{2-7}
		&\multirow{2}*{APSKM}  & IT         &\textbf{3295}&\textbf{6410}&\textbf{9781}&13391\\
		& &CPU        &0.3137 &1.0498 &1.4678 &2.4749 \\
		\bottomrule 
	\end{tabular}}
\end{table}
\begin{table}[!htbp]
	\centering\small
	\caption{IT and CPU of the PSGD, PSKM and APSKM methods for Problem 1 with $C=\{x:\langle A, x\rangle = b\}$, $m=5000$ and $\varepsilon=10^{-4}$}
	\label{table3}
	\scalebox{0.83}{
	\begin{tabular}{cccccccc}
		\toprule  
		\multicolumn{3}{c}{ $\xi$ }  &0.1&0.3&0.5&0.7&0.9\\
		\hline
		\multirow{6}*{$k_c=300$}&\multirow{2}*{PSGD}  & IT  &-&-&27253&46620&112430\\
		& & CPU       &- &- &1.6779 &2.8708 &6.8706 \\
		\cline{2-8}
		&\multirow{2}*{PSKM}  & IT  &6390&6201&7512&7494&1860\\
		& & CPU       &\textbf{0.8447} &2.1523 &2.6123 &2.5642 &0.6620 \\
		\cline{2-8}
		&\multirow{2}*{APSKM}  & IT         &\textbf{3107}&\textbf{2029}&\textbf{1287}&\textbf{693}&\textbf{117}\\
		& &CPU        &1.0001 &\textbf{1.4160} &\textbf{0.9176} &\textbf{0.5032} &\textbf{0.1019} \\
		\midrule
		\multirow{6}*{$k_c=500$}&\multirow{2}*{PSGD}  & IT  &-&-&40275&78502&169179\\
		& & CPU       &- &- &2.7559 &5.3622 &11.4265 \\
		\cline{2-8}
		&\multirow{2}*{PSKM}  & IT  &6966&8419&9328&8231&1862\\
		& & CPU       &\textbf{2.4967 }&2.9682 &0.7258 &0.6372 &0.1514 \\
		\cline{2-8}
		&\multirow{2}*{APSKM}  & IT         &\textbf{3488}&\textbf{3014}&\textbf{1784}&\textbf{1215}&\textbf{124}\\
		& &CPU        &2.5679 &\textbf{2.2439} &\textbf{0.3219} &\textbf{0.2228 }&\textbf{0.0352} \\
		\midrule
		\multirow{6}*{$k_c=1000$}&\multirow{2}*{PSGD}  & IT  &-&-&79079&142208&274670\\
		& & CPU       &- &- &6.1641 &10.7412 &20.9303 \\
		\cline{2-8}
		&\multirow{2}*{PSKM}  & IT  &9208&9286&10695&8724&1894\\
		& & CPU       &\textbf{0.7803} &0.7830 &0.9003 &2.3935 &0.1826 \\
		\cline{2-8}
		&\multirow{2}*{APSKM}  & IT         &\textbf{4675}&\textbf{3710}&\textbf{3245}&\textbf{1637}&\textbf{137}\\
		& &CPU        &0.9159 &\textbf{0.7305} &\textbf{0.6383} &\textbf{0.9301} &\textbf{0.0553 }\\
		\bottomrule 
	\end{tabular}}
\end{table}
\begin{table}[!htbp]
	\centering\small
	\caption{IT and CPU of the PSGD, PSKM and APSKM methods for Problem 2 with $C=\{x:\langle A, x\rangle = b\}$, $m=10000$ and $\varepsilon=10^{-4}$}
	\label{table4}
	\scalebox{0.83}{
	\begin{tabular}{cccccccc}
		\toprule  
		\multicolumn{3}{c}{ $\xi$ }  &0.1&0.3&0.5&0.7&0.9\\
		\hline
		\multirow{6}*{$k_c=300$}&\multirow{2}*{PSGD}  & IT  &40928&35472&38647&55725&18762\\
		& & CPU       &2.5224 &2.1545 &2.4500 &3.3833 &1.1458 \\
		\cline{2-8}
		&\multirow{2}*{PSKM}  & IT  &10988&10661&9274&6993&2099\\
		& & CPU       &0.8504 &0.7951 &0.6914 &0.5262 &0.1611 \\
		\cline{2-8}
		&\multirow{2}*{APSKM}  & IT         &\textbf{3692}&\textbf{1807}&\textbf{1639}&\textbf{859}&\textbf{180}\\
		& &CPU        &\textbf{0.6274} &\textbf{0.3066 }&\textbf{0.2738} &\textbf{0.1465} &\textbf{0.0350} \\
		\midrule
		\multirow{6}*{$k_c=500$}&\multirow{2}*{PSGD}  & IT  &42065&42246&56107&75285&20291\\
		& & CPU       &4.0855 &4.1093 &5.4494 &7.2859 &1.9886 \\
		\cline{2-8}
		&\multirow{2}*{PSKM}  & IT  &11415&10690&9355&7103&2038\\
		& & CPU       &1.2713 &1.2008 &1.0503 &0.8031 &0.2384 \\
		\cline{2-8}
		&\multirow{2}*{APSKM}  & IT         &\textbf{5070}&\textbf{3611}&\textbf{2433}&\textbf{1248}&\textbf{118}\\
		& &CPU        &\textbf{1.0334} &\textbf{0.7475} &\textbf{0.5114} &\textbf{0.2642}&\textbf{0.0411} \\
		\midrule
		\multirow{6}*{$k_c=1000$}&\multirow{2}*{PSGD}  & IT  &61130&72535&88528&108304&21198\\
		& & CPU       &4.6189 &5.4521 &6.6552 &8.1282 &1.6060 \\
		\cline{2-8}
		&\multirow{2}*{PSKM}  & IT  &11549&11170&9778&7273&2081\\
		& & CPU       &\textbf{1.0210} &\textbf{1.0029} &0.8732 &0.6582 &0.2052 \\
		\cline{2-8}
		&\multirow{2}*{APSKM}  & IT         &\textbf{6819}&\textbf{6746}&\textbf{3615}&\textbf{1924}&\textbf{151}\\
		& &CPU        &1.3773 &1.3554 &\textbf{0.7416} &\textbf{0.4027} &\textbf{0.0588} \\
		\bottomrule 
	\end{tabular}}
\end{table}
\section{Conclusions}
We have proposed the NSKM method for solving large-scale nonlinear equations. At each step, only a part of residuals are computed. Furthermore, we have developed two variants of the NSKM method for solving large-scale nonlinear equations with finite convex constraints. Convergence analysis of the proposed methods is given.  

From the numerical point of view, some numerical results show that the NSKM method with the sample of the right size is more effective than the NRK method in terms of CPU times. Moreover, two variants of the NSKM method can converge well to the solution of the constrained nonlinear problems and the APSKM method is superior to the PSKM method for nonlinear problems with large near-linear correlation equality constraints.

\bmhead{Acknowledgments}
This work was supported by the Fundamental Research Funds for the Central Universities(20CX05011A) and the Fundamental Research Funds for the Central Universities (grant number 18CX02041A).

\appendix
\section*{Appendix}
\section{Proof of Lemma \ref{lemm2.3.3} }
\begin{proof}
	Since $f: \mathscr{D}(f)\to \mathbb{R}$ is a convex function, we have that
	\begin{equation}\label{applemma2.3.1}
		f((1-\alpha)x+\alpha y)\leq (1-\alpha)f(x)+\alpha f(y), \forall \alpha \in [0,1], \forall x,y \in \mathscr{D}(f).
	\end{equation}
	By Taylor formula, it holds 
	\begin{equation}\label{applemma2.3.2}
		f((1-\alpha)x+\alpha y)=f(x)+\alpha f'(x)(y-x)+o(\|\alpha(y-x)\|).
	\end{equation}
	Combining \eqref{applemma2.3.1} and \eqref{applemma2.3.2}, we obtain that 
	$$f(y)-f(x)\geq  f'(x)(y-x)+\frac{o(\|\alpha(y-x)\|)}{\alpha}.$$
	Let $\alpha\to 0$, 
	$$f(y)\geq f(x)+ f'(x)(y-x).$$
	This completes the proof.
\end{proof}
\section{Proof of Lemma \ref{lem3.2}}
\begin{proof}
	We first show that $\lambda_k\geq 1$. Observe that 
	\begin{align*}
		&\quad2\langle x_{k-\frac{3}{5}}-x_{k-\frac{1}{5}},x_{k-\frac{3}{5}}-x_{k-\frac{2}{5}}\rangle\\
		&=\|x_{k-\frac{3}{5}}-x_{k-\frac{1}{5}}\|^2+\|x_{k-\frac{3}{5}}-x_{k-\frac{2}{5}}\|^2-\|x_{k-\frac{1}{5}}-x_{k-\frac{2}{5}}\|^2\\
		&\leq 2(\|x_{k-\frac{2}{5}}-x_{k-\frac{3}{5}}\|^2-\|x_{k-\frac{1}{5}}-x_{k-\frac{2}{5}}\|^2),
	\end{align*}
	where the last inequality follows from Lemma \ref{lem3.1}.
	
	Hence $\langle x_{k-\frac{3}{5}}-x_{k-\frac{1}{5}},x_{k-\frac{3}{5}}-x_{k-\frac{2}{5}}\rangle
	\leq \|x_{k-\frac{2}{5}}-x_{k-\frac{3}{5}}\|^2$, which implies that 
	\begin{equation}\label{equa12}
		\lambda_k=\frac{\|x_{k-\frac{3}{5}}-x_{k-\frac{2}{5}}\|^2}{\left<x_{k-\frac{3}{5}}-x_{k-\frac{1}{5}}, x_{k-\frac{3}{5}}-x_{k-\frac{2}{5}}\right>}\geq 1.
	\end{equation}
	
	Next, we will prove that $x_{k}-x_{k-\frac{2}{5}}$ and $x_{k-\frac{3}{5}}-x_{k-\frac{2}{5}}$ are orthogonal.
	\begin{align}\label{equa13}
		\langle x_{k}-x_{k-\frac{2}{5}}, x_{k-\frac{3}{5}}-x_{k-\frac{2}{5}} \rangle &=\langle (x_{k-\frac{3}{5}}-x_{k-\frac{2}{5}})+\lambda_k (x_{k-\frac{1}{5}}-x_{k-\frac{3}{5}}), x_{k-\frac{3}{5}}-x_{k-\frac{2}{5}} \rangle \nonumber \\
		&=\|x_{k-\frac{3}{5}}-x_{k-\frac{2}{5}}\|^2+\lambda_k \langle x_{k-\frac{1}{5}}-x_{k-\frac{3}{5}}, x_{k-\frac{3}{5}}-x_{k-\frac{2}{5}} \rangle \nonumber \\
		&=\|x_{k-\frac{3}{5}}-x_{k-\frac{2}{5}}\|^2(1+\frac{\langle x_{k-\frac{1}{5}}-x_{k-\frac{3}{5}}, x_{k-\frac{3}{5}}-x_{k-\frac{2}{5}} \rangle}{\left<x_{k-\frac{3}{5}}-x_{k-\frac{1}{5}}, x_{k-\frac{3}{5}}-x_{k-\frac{2}{5}}\right>}) \nonumber \\
		&=0.
	\end{align}
	
	Finally, we utilize \eqref{equa12} and \eqref{equa13} to prove \eqref{lemma3.2.1}.
	
	For every $x\in C_{\alpha_1^k} \cap C_{\alpha_2^k}$, we have 
	$$\|x_k-x\|^2=\|x_k-x_{k-\frac{1}{5}}\|^2+\|x_{k-\frac{1}{5}}-x\|^2+2\langle x_k-x_{k-\frac{1}{5}}, x_{k-\frac{1}{5}}-x\rangle.$$ 
	
	By writing $\langle x_k-x_{k-\frac{1}{5}}, x_{k-\frac{1}{5}}-x \rangle=\langle x_k-x_{k-\frac{1}{5}}, x_{k-\frac{1}{5}}-x_k \rangle+\langle x_k-x_{k-\frac{1}{5}}, x_{k}-x \rangle$, we find that
	$$\|x_k-x\|^2=\|x_{k-\frac{1}{5}}-x\|^2-\|x_{k}-x_{k-\frac{1}{5}}\|^2+2\langle x_k-x_{k-\frac{1}{5}}, x_{k}-x\rangle.$$
	
	By the definition of $x_{k}$, we obtain that 
	\begin{align*}
		\langle x_k-x_{k-\frac{1}{5}}, x_{k}-x\rangle&=(1-\lambda_k)\langle x_{k-\frac{3}{5}}-x_{k-\frac{1}{5}}, x_{k}-x\rangle\\
		&=(1-\lambda_k)(\langle x_{k-\frac{3}{5}}-x_{k-\frac{2}{5}}, x_{k}-x\rangle+\langle x_{k-\frac{2}{5}}-x_{k-\frac{1}{5}}, x_{k}-x\rangle).
	\end{align*}

	For the first inner product of the above formula, we have
	$$\langle x_{k-\frac{3}{5}}-x_{k-\frac{2}{5}}, x_{k}-x\rangle=\langle x_{k-\frac{3}{5}}-x_{k-\frac{2}{5}}, x_{k}-x_{k-\frac{2}{5}}\rangle+\langle x_{k-\frac{3}{5}}-x_{k-\frac{2}{5}}, x_{k-\frac{2}{5}}-x\rangle \geq 0,$$
	where the inequality comes from Lemma \ref{lem3.1} and \eqref{equa13}.
	
	For the second inner product, we can obtain
	\begin{align*}
		\langle x_{k-\frac{2}{5}}-x_{k-\frac{1}{5}}, x_{k}-x\rangle&=\langle x_{k-\frac{2}{5}}-x_{k-\frac{1}{5}}, x_{k}-x_{k-\frac{1}{5}}\rangle+\langle x_{k-\frac{2}{5}}-x_{k-\frac{1}{5}}, x_{k-\frac{1}{5}}-x\rangle\\
		&\geq \langle x_{k-\frac{2}{5}}-x_{k-\frac{1}{5}}, x_{k}-x_{k-\frac{1}{5}}\rangle\\
		&=(1-\lambda_k)\langle x_{k-\frac{2}{5}}-x_{k-\frac{1}{5}}, x_{k-\frac{3}{5}}-x_{k-\frac{1}{5}}\rangle\\
		&\geq0,
	\end{align*}
	where the first inequality follows Lemma \ref{lem3.1}, the second equality comes from the definition of $x_k$ and the second ineuality is from Lemma \ref{lem3.1} and \eqref{equa12}.
	
	Thus, 
	\begin{equation}\label{lemm3.2.1}
		\|x_k-x\|^2\leq \|x_{k-\frac{1}{5}}-x\|^2-\|x_k-x_{k-\frac{1}{5}}\|^2 \leq \|x_{k-\frac{1}{5}}-x\|^2.
	\end{equation}

	Since the iteration points $x_i$ $(i=k-\frac{3}{5},k-\frac{2}{5},k-\frac{1}{5})$ are obtained by projecting on the closed convex sets, by Lemma \ref{lem3.1}, it results in
	$$\|x_i-x\|^2\leq \|x_{i-\frac{1}{5}}-x\|^2-\|x_i-x_{i-\frac{1}{5}}\|^2.$$
	
	Thus,
	\begin{equation}\label{lemma3.2.2}
		\|x_{k-\frac{1}{5}}-x\|^2\leq\|x_{k-\frac{2}{5}}-x\|^2\leq\|x_{k-\frac{3}{5}}-x\|^2\leq \|x_{k-\frac{4}{5}}-x\|^2.
	\end{equation} 

	From \eqref{lemm3.2.1} and \eqref{lemma3.2.2}, we get that
	$$\|x_k-x\|^2\leq\|x_{k-\frac{4}{5}}-x\|^2.$$
\end{proof}




\bibliographystyle{plain}
\bibliography{reference.bib}


\end{document}